\documentclass[final,onefignum,onetabnum]{siamart190516}

\usepackage{amssymb, mathtools, tikz}
\usetikzlibrary{shapes, patterns}
\makeatletter 
    \def\imod#1{\allowbreak\mkern10mu({\operator@font mod}\,\,#1)}
\makeatother

\newsiamremark{remark}{Remark}
\newsiamremark{hypothesis}{Hypothesis}
\crefname{hypothesis}{Hypothesis}{Hypotheses}
\newsiamthm{claim}{Claim}

\newsiamremark{strat}{Strategy}
\newsiamremark{case}{Case}

\headers{The Localization Game}{J. Boshoff, A. Roux} 

\title{The Localization Game On Cartesian Products}

\author{Jeandré Boshoff\thanks{Department of Applied Mathematics, Stellenbosch University, Stellenbosch, ZA 
  (\email{jboshoff@sun.ac.za}, \email{rianaroux@sun.ac.za}).\funding{This work is based on the research supported by the National Research Foundation of South Africa (Grant number: 121931).}}
\and Adriana Roux\footnotemark[1]}

\ifpdf
\hypersetup{
  pdftitle={The Localization Game On Cartesian Products},
  pdfauthor={J. Boshoff and A. Roux}
}
\fi

\begin{document}

\maketitle

\begin{abstract}
    The localization game is played by two players: a Cop with a team of $k$ cops, and a Robber.
    The game is initialised by the Robber choosing a vertex $r \in V$, unknown to the Cop.
    Thereafter, the game proceeds turn based.
    At the start of each turn, the Cop probes $k$ vertices and in return receives a distance vector.
    If the Cop can determine the exact location of $r$ from the vector, the Robber is located and the Cop wins.
    Otherwise, the Robber is allowed to either stay at $r$, or move to $r'$ in the neighbourhood of $r$.
    The Cop then again probes $k$ vertices.
    The game continues in this fashion, where the Cop wins if the Robber can be located in a finite number of turns.
    The localization number $\zeta(G)$, is defined as the least positive integer $k$ for which the Cop has a winning strategy irrespective of the moves of the Robber.
    In this paper, we focus on the game played on Cartesian products.
    We prove that $\zeta( G \square H) \geq \max\{\zeta(G), \zeta(H)\}$ as well as $\zeta(G \square H) \leq \zeta(G) + \psi(H) - 1$ where $\psi(H)$ is a doubly resolving set of $H$.
    We also show that $\zeta(C_m \square C_n)$ is mostly equal to two.
\end{abstract}



\section{Introduction}
    The localization game is played on a simple, connected, undirected graph $G=(V,E)$.
    Two players are involved in this game: a Cop who has a team of $k$ cops, and a Robber.
    To start the game, the Robber chooses a vertex $r \in V$, unknown to the Cop.
    After this, the game proceeds turn based.
    
    At the start of each turn, the Cop probes $k$ vertices $B = \lbrace b_1, b_2, \ldots, b_k \rbrace $.
    In return, the Cop receives the vector $\vec{D}(\{r\}, B) = [d_1, d_2, \ldots, d_k]$ where $d_i = d_G(r,b_i)$ is the distance in $G$ from $r$ to $b_i$ for $i = 1, 2, \ldots, k$.
    If the Cop can determine the exact location of $r$ from $\vec{D}(\{r\}, B)$, the Robber is located and the Cop wins.
    Otherwise, the Robber is allowed to either stay at $r$, or move to $r'$ in the neighbourhood $N[r]$ of $r$.
    The Cop then again probes $k$ vertices.
    These $k$ vertices are allowed to be the same as in previous turns.
    The game continues in this fashion, where the Cop wins if the Robber can be located in a finite number of turns.
    If the Cop fails to locate the Robber in a finite number of turns, the Robber wins.
    The localization number $\zeta(G)$, is defined as the least positive integer $k$ for which the Cop has a winning strategy irrespective of the moves of the Robber.
    Therefore if less than $\zeta(G)$ cops are used to play the game on $G$, it is possible that the Cop never locates the Robber.
    Thus an avoidance strategy for the Robber entails proving that for any sequence of probes by the Cop, the distance vector to the Robber's location at any given turn is not unique.
    
    A version of the localization game with only one cop was introduced by Seager in 2012 \cite{1cop} and studied further in \cite{1copBrandt}, \cite{1copCarraher} and \cite{1copseager}. 
    The localization game itself is a variant of the game of Cops and Robbers and was introduced independently by Bosek et al. \cite{Localization} and Haslegrave et al. \cite{Independent} in 2018.
    In the same year, the game was further studied by Bonato et al. \cite{Newlocbounds} and another variation of the game was introduced by Bosek et al. \cite{Centroidal}.
    In this paper we provide upper and lower bounds on the Cartesian product of general graphs.
    We give special consideration to the product of cycles and show that the localization number of nearly all products of cycles is two.
    
    This paper is organized as follows. 
    In the next section, we give some basic results.
    Lower bounds for $G \square H$ are given in \cref{sec: lowbound} and an upper bound is given in \cref{sec: upbound}.
    In \cref{sec: cyclicgrid} we calculate the localization number of the Cartesian product of two cycles.
    
\section{Basic results}
    The localization number is related to the metric dimension of a graph.
    To define the metric dimension, we start with a resolving set:
    \begin{definition}[Resolving set \cite{Dimensions}]
        A set of vertices $S \subseteq G$ is a resolving set of graph $G$ if every vertex in $G$ is uniquely defined by its distance to the vertices in $S$.
    \end{definition}
    \begin{definition}[Metric dimension \cite{Dimensions}]
        The metric dimension  $\dim(G)$ of a graph $G$ is defined as the minimum cardinality of a set $S \subseteq G$ such that $S$ resolves $G$.
    \end{definition}
    Note that $\dim(G)$ can equivalently be defined as the smallest positive integer $k$ such that the Cop locates the Robber in one turn and hence
    \begin{equation} \label{eq: zeta-dim bound}
        \zeta(G) \leq \dim(G) \leq n-1
    \end{equation}
    where $n$ is the order of the graph.
    The localization number and the metric dimension may be equal, for example, $\zeta(P_n) = \dim(P_n) = 1$ and $\zeta(K_n) = \dim(K_n) = n-1$.
    However, the difference between these two parameters can be arbitrarily large.
    Bosek et al. \cite{Localization} showed that $\zeta(K_{2,n}) = 2$, but it is known that $\dim(K_{2,n}) = n$ for $n > 2$.
    
    \begin{definition}[Hideout \cite{1copseager}]
        A hideout is defined as a subgraph $H$ of $G$ where the robber can win by remaining on the vertices of $H$.
    \end{definition}
    \begin{lemma}[\cite{1copCarraher}] \label{lem: loc not 1}
        Let $G$ be any graph containing a cycle of length at most five, where the localization game is played with one cop.
        Then this cycle is a hideout such that $\zeta(G) \neq 1$.
    \end{lemma}
    
    The remainder of the article will focus on the localization game played on the Cartesian product of two graphs.
    \begin{definition}[Cartesian product]
        The Cartesian product $G \square H$ of two graphs $G$ and $H$ is a graph with vertex set the Cartesian product $V(G) \times V(H)$.
        Further two vertices $(u,u')$ and $(v,v')$ in $G \square H$ are adjacent if and only if either $u=v$ and $d_H(u',v')=1$, or $u'=v'$ and $d_G(u,v)=1$.
    \end{definition}
    For the graph $G \square H$ where $G$ has order $m$ and $H$ has order $n$, label the vertices $v_{i,j}$ for $i \in \{0,1,\ldots,m-1 \}$ and $j \in \{0,1,\ldots,n-1 \}$ such that $v_{0,0}$ is the bottom left vertex and the grid is embedded on the positive quadrant of a Cartesian coordinate system.
    The indices $i$ and $j$ of $v_{i,j}$ will be calculated modulo $m$ and $n$ respectively.
    
    \begin{lemma}[\cite{dimLandmarks}] \label{lem: grid dim}
        For $d \geq 2$, the metric dimension of a $d$-dimensional grid is $d$.
    \end{lemma}
    \begin{proposition}
        Let $G_{m,n}$ be the Cartesian product of two paths, with $m,n \geq 2$.
        Then $\dim (G_{m,n}) = \zeta(G_{m,n}) = 2$.
    \end{proposition}
    \begin{proof}
        By \cref{lem: grid dim}, the dimension of $G_{m,n}$ is two and therefore $\zeta(G_{m,n}) \leq 2$.
        Note that the grid $G_{m,n}$ contains a cycle of length four and thus by \cref{lem: loc not 1}, $\zeta(G_{m,n}) \geq 2$.
    \end{proof}
    The imagination strategy introduced by Brešar et al. \cite{Domination} is a technique used to find bounds on parameters concerning games on graphs.
    The idea of the imagination strategy is that one of the players imagines another appropriate game and plays in it according to a known winning strategy.
    As an example, say the localization game is played on some graph $G$.
    Assume the Cop plays by using the imagination strategy, where a graph $G'$ is imagined such that a winning strategy is known for the Cop on graph $G'$.
    The Cop therefore has a probe $B_1'$ on graph $G'$ which will lead to the Cop locating the Robber in a finite number of turns.
    This probe is copied to $G$ such that the Cop probes $B_1$ in the first turn.
    The Cop next receives some distance vector $\vec{D}(B_1,r)$ and copies this to graph $G'$.
    Again a second probe $B_2'$ on $G'$ is known, which is copied to the graph $G$ such that $B_2$ is probed.
    The game continues in this fashion.
    It is possible that a probe by the Cop in the imagined game is not legal in the real game and it is also possible that the distance received by the Cop in the real game does not exist in the imagined game.
    Both these problems need to be considered when using this strategy.
    
\section{Lower bounds for Cartesian products} \label{sec: lowbound}
    In this section we give two lower bounds for the Cartesian product of two general graphs.
    When considering Cartesian products, projections provide us with a way to move between the product and the individual graphs. 
    \begin{definition}[Projections \cite{Dimensions}]
        Let $S$ be a set of vertices in the Cartesian product $G \square H$.
        The projection of $S$ onto $G$ is the set of vertices $v \in V(G)$ for which there exists a vertex $(v,v') \in S$.
        Similarly, the projection of $S$ onto $H$ is the set of vertices $v' \in V(H)$ for which there exists a vertex $(v,v') \in S$.
    \end{definition}    
    
    Since the Cartesian product of two connected graphs of orders at least two always contains a 4-cycle, the following lower bound follows from \cref{lem: loc not 1}:
    \begin{proposition} \label{prop: no loc 1 product}
        Let $G$ and $H$ be any connected graphs of orders at least two. 
        Then $\zeta(G \square H) \geq 2$.
    \end{proposition}
    
    The following lemma provides a link between the resolving set of the product $G \square H$ and the resolving set of $G$ or $H$:
    \begin{lemma}[\cite{Dimensions}] \label{lem: subset prod}
        Let $S \subseteq V(G \square H)$ for graphs $G$ and $H$.
        Then every pair of vertices in a fixed column of $G \square H$ is uniquely defined by their distance to the vertices in $S$ if and only if the projection of $S$ onto $H$ uniquely defines all vertices in $H$ by their distance to the projection.
        Similarly, every pair of vertices in a fixed row of $G \square H$ is uniquely defined by their distance to the vertices in $S$ if and only if the projection of $S$ onto $G$ uniquely defines all vertices in $G$ by their distance to the projection.
    \end{lemma}
    
    By making use of the imagination strategy we can show that the localization number of the product of graphs $G$ and $H$ is at least the maximum of the localization number of $G$ and $H$:
    \begin{theorem} \label{thm: loc(GH)>loc(G)}
        For any two graphs $G$ and $H$, the following equation holds: \\ $\zeta( G \square H) \geq \max\{\zeta(G), \zeta(H)\}$. 
    \end{theorem}
    \begin{proof}
        Consider the localization game played on the Cartesian product $G \square H$, where $G$ and $H$ are any two graphs.
        Say the Cop plays with $k=\zeta(G)-1$ cops and that the Robber plays by imagining the localization game on $G$.
        In the first turn, the Robber occupies some vertex $r_0$ in the imagined game.
        In the real game, the Robber chooses to occupy vertex $(r_0,j)$ for some row $j$ in $G \square H$.
        In the turns to follow, the Robber applies the following strategy:
        Say in turn $\alpha$ the Cop probes $B_{\alpha} = \{b_1, b_2, \ldots, b_k \}$.
        Let $S_{\alpha}$ be the projection of $B_{\alpha}$ onto $G$, such that $S_{\alpha}$ contains at most $k$ vertices.
        The Robber then imagines the Cop probes $S_{\alpha}$ on graph $G$, where the Robber is always able to avoid capture since $|S_{\alpha}| \leq k < \zeta(G)$.
        Therefore after probe $S_{\alpha}$ on $G$, there exists a vertex $r_{\alpha}$ where the Robber is safe.
        After probe $B_{\alpha}$ in the real game, the Robber will then be safe at vertex $(r_{\alpha}, j)$ by \cref{lem: subset prod}.
        The games continues in this fashion such that the Cop never wins and $\zeta(G \square H) > k = \zeta(G)-1$.
        In a similar fashion it can be shown that $\zeta(G \square H) > \zeta(H)-1$ and thus $\zeta( G \square H) \geq \max\{\zeta(G), \zeta(H)\}$.
    \end{proof}
    This lower bound is reached for $C_{2p+1} \square C_n$ if $p \geq 3$ and $4 \leq n \leq 6$.
    This is proven in \cref{sec: cyclicgrid} where we show that $\zeta(C_{2p+1}) = 1$ and $\zeta(C_n) = 2$, but $\zeta(C_{2p+1} \square C_n) = 2$.
    
\section{An upper bound for Cartesian products} \label{sec: upbound}
    This section gives an upper bound to the Cartesian product of two general graphs.
    After the Cop probes $k$ vertices there will be some vertices which are the same distance away from the probe. 
    \begin{definition}[Safe vertex]
        A vertex $v$ is called a safe vertex if it is not uniquely defined by probe $B$.
        In other words, there exists another vertex $w$ that is the same distance from $B$ as $v$.
    \end{definition}
    \begin{definition}[Safe set]
        A safe set is a set of safe vertices that are all the same distance from $B$.
        By definition, every safe vertex is part of a safe set.
    \end{definition}
    \begin{definition}[Robber set \cite{1copseager}]
        The robber set is defined as the safe set that the Robber has been localized to and is denoted by $O_{\alpha}$ in turn $\alpha$.
        Therefore if the robber set only contains one vertex, the Cop wins.
        If not, the Cop requires another probe and the Robber can move to any vertex in $N[O_{\alpha}]$.
    \end{definition}
    
    A strong form of a resolving set is needed to find an upper bound on the localization number of two graphs:
    \begin{definition}[Doubly resolving sets \cite{Dimensions}]
        Let $G \neq K_1$ be a graph.
        Two vertices $v_1, v_2 \in V(G)$ are doubly resolved by vertices $u_1,u_2 \in V(G)$ if $$d(v_1,u_1)-d(v_2,u_1) \neq d(v_1,u_2)-d(v_2,u_2).$$
        A set $W \subseteq V(G)$ doubly resolves $G$ and is a doubly resolving set, if every pair of distinct vertices $v_1, v_2 \in V(G)$ are doubly resolved by two vertices in $W$.
        The doubly resolving set with the smallest cardinality is denoted by $\psi(G)$.
    \end{definition}
    Even though $\psi(G)$ is defined in  \cite{Dimensions}, it is never named and hence we name it the \emph{doubly resolving number} of a graph $G$.
    Every graph $G$ with at least two vertices has a doubly resolving set and therefore it is well defined.
    Note that when calculating if some set $W \subseteq V(G)$ is a doubly resolving set, the vertex pairs inside $W$ need not be considered.
    To prove this, consider any two distinct vertices $w_1,w_2 \in W$.
    Clearly $d(w_1,w_1)-d(w_2,w_1) = -d(w_2,w_1)$ where $d(w_1,w_2) - d(w_2,w_2) = d(w_1,w_2)$ such that $w_1,w_2$ are doubly resolved by $W$.
    Also note that it is easy to check that every doubly resolving set is also a resolving set.
    Cáceres et al. proved that $2 \leq \psi(G) \leq m-1$ for any graph $G$ of order $m \geq 3$, where it was also shown that $\dim(G) \leq \psi(G)$.
    They also proved the following proposition:
    \begin{proposition}[\cite{Dimensions}] \label{prop: dim(GH) upbound}
        For all graphs $G$ and $H \neq K_1$, $\dim(G \square H) \leq \dim(G) + \psi(H)-1$.
    \end{proposition}
    
    Since $\zeta(G) \leq \dim(G)$, this proposition provides an upper bound to $\zeta(G \square H)$.
    However, by making use of similar arguments, we can improve the bound:
    \begin{theorem} \label{thm: general upbound}
        Let $G$ and $H$ be any connected graphs.
        Then $\zeta(G \square H) \leq \zeta(G) + \psi(H) - 1$. 
    \end{theorem}
    \begin{proof}
        It needs to be shown that the Cop can win on $G \square H$ using $\kappa$ cops, where $\kappa=\zeta(G) + \psi(H) - 1$.
        To this end, the Cop imagines the localization game on graph $G$.
        Let $T$ be a doubly resolving set of $H$ such that $\psi(H) = |T|$.
        Further, say the Cop probes $B_1$ in the first turn of the imagined game such that $|B_1| = \zeta(G)$.
        For a fixed $b_1 \in B_1$ and $t \in T$, define a set $X_1$ such that $X_1 \coloneqq \{(b_1,t^i) : t^i \in T \} \cup \{(b_1^i,t) : b_1^i \in B_1 \}$.
        Note that $|X_1|=\kappa$ and each entry of $X_1$ is a vertex in $G \square H$.
        In the first turn in the real game, the Cop probes $X_1$.
        It will now be shown that any safe set for this probe is contained in a single row of $G \square H$ and further that the projection of this safe set onto $G$ is a valid safe set in $G$.
        
        Consider two distinct vertices $(g,h)$ and $(g',h')$ of $G \square H$ where
        \begin{equation} \label{eq: gh-relat}
            \vec{D}\left( (g,h),X_1 \right) = \vec{D}\left( (g',h'),X_1 \right).
        \end{equation}
        Since $T$ is a doubly resolving set, by \cref{lem: subset prod} the projection of $X_1$ onto $H$ uniquely defines all vertices in $H$.
        Hence if $g=g'$, Equation \cref{eq: gh-relat} only holds if $h=h'$. 
        
        Now consider the case where $g \neq g'$ and assume $h \neq h'$.
        Since $T$ is a doubly resolving set of $H$, there exists two vertices $t_k,t_l \in T$ such that
        \begin{equation} \label{eq: T equation}
            d_H(h,t_k) - d_H(h',t_k) \neq d_H(h,t_l) - d_H(h',t_l).
        \end{equation}
        Equation \cref{eq: gh-relat} implies that 
        \begin{equation*}
            d_{G \square H} \left( (g,h), (x,x') \right) = d_{G \square H} \left( (g',h'), (x,x') \right)
        \end{equation*}
        for any $(x,x') \in X_1$.
        Thus 
        \begin{align*}
            d_{G \square H} \left( (g,h), (b_1,t_k) \right) &= d_{G \square H} \left( (g',h'), (b_1,t_k) \right) \text{ and } \\
            d_{G \square H} \left( (g,h), (b_1,t_l) \right) &= d_{G \square H} \left( (g',h'), (b_1,t_l) \right)
        \end{align*}
        such that 
        \begin{align}
            d_G(g,b_1)+d_H(h,t_k) &= d_G(g',b_1) + d_H(h',t_k) \text{ and } \label{eq: break1} \\
            d_G(g,b_1)+d_H(h,t_l) &= d_G(g',b_1) + d_H(h',t_l). \label{eq: break2} 
        \end{align}
        Equations \cref{eq: break1} and \cref{eq: break2} together imply
        \begin{equation*}
            d_H(h,t_k) - d_H(h',t_k) = d_H(h,t_l) - d_H(h',t_l),
        \end{equation*}
        contradicting Equation \cref{eq: T equation} and therefore Equation \cref{eq: gh-relat} only holds if $h=h'$.
        It follows that $d_G(g,b_1) = d_G(g',b_1)$ such that vertices $(g,h)$ and $(g',h')$ are in the same safe set in $G \square H$ if and only if vertices $g$ and $g'$ are in the same safe set in the imagination game.
        
        Say the Robber is localized to robber set $O_1$ in $G \square H$, where $Q_1$ is the projection of $O_1$ onto $G$.
        It has been shown that $O_1$ is contained in a single row and that $Q_1$ is a valid robber set in the imagination game.
        For robber set $Q_1$ in the imagination game, a probe $B_2$ is known such that the Cop wins in a finite number of turns.
        For a fixed $b_2 \in B_2$, let $X_2 \coloneqq \{(b_2,t^i): t^i \in T \} \cup \{(b_2^i,t) : b_2^i \in B_2 \}$ such that $|X_2| = \kappa$.
        As before two vertices $(a,b)$ and $(a',b')$ in $N[O_1]$ only belong to the same safe set in the real game if $b=b'$ and if $a$ and $a'$ belong to the same safe set in the imagination game.
        Say the robber is localized to $O_2$ in the real game and localized to $Q_2$ in the imagination game.
        Then $O_2$ will be contained in a single row and its projection onto $G$ will be equal to $Q_2$.
        Therefore the Cop can imagine the robber set $Q_2$ on $G$ such that $B_3$ is probed.
        The Cop continues in this fashion until the Robber is located.
        This is guaranteed because in some turn $s$ on graph $G$, the robber set $Q_s$ will only contain one vertex and therefore the robber set $O_s$ in the real game will also only contain one vertex. 
    \end{proof}
    \begin{corollary}
        Let $G$ and $H$ be any connected graphs.
        By restricting $\zeta(G)$ or $\psi(H)$, we get the following results:
        \begin{enumerate}
            \item If $\zeta(G)=1$, then $\zeta(H) \leq \zeta(G \square H) \leq \psi(H)$.
            \item If $\psi(H)=2$, then $\zeta(G) \leq \zeta(G \square H) \leq \zeta(G)+1$.
            \item If $\zeta(G)=1$ and $\psi(H)=2$, then $\zeta(G \square H) = 2$.
        \end{enumerate}
    \end{corollary}
    
\section{Products of cycles} \label{sec: cyclicgrid}
    Let $C_m \square C_n$ be the Cartesian product of two cycles of order $m$ and $n$ respectively.
    
    \begin{definition}[Second difference]
        For probe $B$, vertex $v$ and distance vector $\vec{D}(B,v) = [a,b]$, we define the second difference $DD$ as $DD(B,v) = b-a$.
    \end{definition}
    Clearly, if $DD(B,x) \neq DD(B,y)$, then $\vec{D}(B,x) \neq \vec{D}(B,y)$ for any two vertices $x,y$.
    For each second difference that is not unique to a single vertex, there exists a set of vertices where the Robber is potentially safe.
    This set will be called a safe house.
    \begin{definition}[Safe house]
        For a graph $G$ with probe $B$, a safe house $S_h$  is the set of all vertices $v \in V(G)$ such that $DD(B,v)=h$.
        Note that safe sets are confined to a specific safe house.
    \end{definition}
    \begin{definition}[Cop house]
        Let $G$ be a graph where the Cop probes $B_{\alpha}$ in turn $\alpha$.
        A cop house is a subset of $V(G)$ that contains only vertices from different safe sets.
    \end{definition}
    A cop house is therefore ``locally unique": if the Robber is restricted to movement in a cop house in turn $\alpha$, the Cop wins immediately.
    Note that a cop house may contain safe vertices, but all vertices in a cop house belong to different safe sets.
    
    \begin{definition}[Diagonal safe pair]
        A diagonal safe pair is a safe set that contains two safe vertices that can be written as $\{v_{a,b}, v_{a+1,b+1}\}$ (positive diagonal) or $\{v_{a,b}, v_{a+1,b-1}\}$ (negative diagonal) for integers $a \text{ and } b$.
    \end{definition}
    \begin{definition}[Horizontal safe pair]
        A horizontal safe pair $S_d^h$ is a safe set that contains two safe vertices a distance of $d$ apart that can be written as $\{v_{a,b}, v_{a+d,b} \}$.
    \end{definition}
    \begin{definition}[Vertical safe pair]
        A vertical safe pair $S_d^v$ is a safe set that contains two safe vertices a distance of $d$ apart that can be written as $\{v_{a,b}, v_{a,b+d} \}$.
    \end{definition}
    
    The main result in this section shows that $\zeta(C_m \square C_n) = 2$ for most cases of $m$ and $n$.
    \begin{theorem} \label{thm: cyclic grid loc}
        Let $C_m \square C_n$ be a product of cycles with $m,n$ integers such that $m\geq n \geq 3$.
        If $m=n=3$ or if $m$ is even while $n=4$, then $\zeta(C_m \square C_n)=3$.
        Otherwise, $\zeta(C_m \square C_n)=2$.
    \end{theorem}
    From \cref{thm: general upbound} we have the following result for cycles:
    \begin{equation} \label{eq: cycle upbound}
        \zeta(C_m \square C_n) \leq \zeta(C_m) + \psi(C_n) - 1.
    \end{equation}
    From \cite{1copCarraher} and \cite{1copseager} it follows that
    \begin{equation*}
        \zeta(C_m) = 
        \begin{cases}
            1 & \text{for } m \geq 7 \\
            2 & \text{for } m \leq 6.
        \end{cases}
    \end{equation*}
    Further, Cáceres et al. \cite{Dimensions} showed that
    \begin{equation*}
        \psi(C_n) = 
        \begin{cases}
            2 & \text{for odd $n$} \\
            3 & \text{for even $n$}
        \end{cases}
    \end{equation*}
    such that $\zeta(C_m \square C_n) = 2$ for $m \geq 7$ and $n$ odd.
    The value of $\zeta(C_m \square C_n)$ for $m \leq 6$ or when $n$ is even will be determined in three separate cases: the product of two odd cycles, an even and an odd cycle and lastly two even cycles.
    
    \subsection{Odd by Odd}
        First consider the localization number of $C_m \square C_n$ where $m$ and $n$ are odd and $m \geq n$.
        Since $n$ is odd, it is known that $\zeta(C_m \square C_n) = 2$ when $m \geq 7$ and therefore only two cases for $m$ are considered here: $m=3$ and $m=5$.
        For $m = n = 3$, we prove the following result:
        \begin{proposition} \label{prop: C3xC3}
            Let $C_3 \square C_3$ be a product of cycles.
            Then $\zeta(C_3\square C_3) = 3$.
        \end{proposition}
        \begin{proof}
            Note that $\zeta(C_3) = \psi(C_3) = 2$ and therefore $\zeta(C_3\square C_3) \leq 3$ by Equation \cref{eq: cycle upbound}. 
            It follows that we only need to show that there exists a winning strategy for the Robber if only two cops are used.
            To this end, say the Cop probes $B_1=\{b_1,b_2 \}$ in the first turn and let
            \begin{equation*}
                Z=V(C_3 \square C_3) \setminus B_1
            \end{equation*}
            be the vertices not probed by the Cop.
            Since $\operatorname{diam}(C_3 \square C_3)=2$, the distance vector $\vec{D}(B_1,z)$ for $z \in Z$ may be one of four unique distance vectors.
            Since $|Z|=7$, there exists safe vertices and the Robber can avoid capture in the first turn.
            Say $u$ and $v$ are two vertices in the same safe set and the Robber is at one of these two vertices.
            Then $|N[\{u,v\}]| \geq 5$ and again by the pigeonhole principle, at least two vertices in $N[\{u,v\}]$ are not uniquely defined by $B_2$.
            Thus, at any turn, there are at least two vertices where the Robber is safe, irrespective of the Cop's probe, and therefore $\zeta(C_3 \square C_3) \geq 3$.
        \end{proof}
        Now consider the case when $m=n=5$.
        \begin{proposition} \label{prop: loc(C5xC5)}
            Let $C_5 \square C_5$ be the product of two cycles.
            Then $\zeta(C_5 \square C_5)=2$.
        \end{proposition}
        \begin{proof}
            The Cop probes $B_1 = \{v_{2,4}, v_{2,2} \}$ in the first turn.
            For any vertex $v_{i,j}$, the distance vector $\vec{D}(B_1, v_{i,j})$ is given in \cref{fig: OxO B1}.
            Safe houses are indicated with the same colour and vertices that belong to the same safe set have the same shape and colour.
            The probed vertices are indicated as squares and empty vertices do not form part of a safe set.
            The distance from a vertex to $B_1$ is indicated above the vertex.
            From the figure it can be seen that all safe sets have the form $\{v_{i,j}, v_{4-i,j} \}$ for $i,j = 0,1,2,3,4$.
            Also, notice the presence of a cop house in columns $0$ to $2$.
            \begin{figure}[h]
                \centering
                \begin{tikzpicture}[scale=1.2]
                    \draw[fill=black!20] (0.55,-0.3) rectangle (3.4,4.65);
                    \node[rotate=90] at (0.3,2) {Cop house};
                
                    \node[draw, regular polygon, regular polygon sides=5, minimum size=16, fill=cyan, label=north: {\footnotesize $[3,4]$}] at (1,0) {};
                    \node[draw, regular polygon, regular polygon sides=5, minimum size=16, fill=cyan, label=north: {\footnotesize $[3,4]$}] at (5,0) {};
                    \node[draw, star, star points=6, fill=cyan, label=north: {\footnotesize $[2,3]$}] at (2,0) {};
                    \node[draw, star, star points=6, fill=cyan, label=north: {\footnotesize $[2,3]$}] at (4,0) {};
                    \node[draw, minimum size=16, circle, fill=white, label=north: {\footnotesize $[1,2]$}] at (3,0) {};
                    
                    \node[draw, regular polygon, regular polygon sides=5, minimum size=16, fill=black, label=north: {\footnotesize $[4,3]$}] at (1,1) {};
                    \node[draw, regular polygon, regular polygon sides=5, minimum size=16, fill=black, label=north: {\footnotesize $[4,3]$}] at (5,1) {};
                    \node[draw, star, star points=6, fill=black, label=north: {\footnotesize $[3,2]$}] at (2,1) {};
                    \node[draw, star, star points=6, fill=black, label=north: {\footnotesize $[3,2]$}] at (4,1) {};
                    \node[draw, minimum size=16, circle, fill=white, label=north: {\footnotesize $[2,1]$}] at (3,1) {};
                    
                    \node[draw, regular polygon, regular polygon sides=5, minimum size=16, fill=olive, label=north: {\footnotesize $[4,2]$}] at (1,2) {};
                    \node[draw, regular polygon, regular polygon sides=5, minimum size=16, fill=olive, label=north: {\footnotesize $[4,2]$}] at (5,2) {};
                    \node[draw, star, star points=6, fill=olive, label=north: {\footnotesize $[3,1]$}] at (2,2) {};
                    \node[draw, star, star points=6, fill=olive, label=north: {\footnotesize $[3,1]$}] at (4,2) {};
                    \node[draw, minimum size=16, fill=white, label=north: {\footnotesize $[2,0]$}] at (3,2) {};
                    
                    \node[draw, regular polygon, regular polygon sides=5, minimum size=16, fill=orange, label=north: {\footnotesize $[3,3]$}] at (1,3) {};
                    \node[draw, regular polygon, regular polygon sides=5, minimum size=16, fill=orange, label=north: {\footnotesize $[3,3]$}] at (5,3) {};
                    \node[draw, star, star points=6, fill=orange, label=north: {\footnotesize $[2,2]$}] at (2,3) {};
                    \node[draw, star, star points=6, fill=orange, label=north: {\footnotesize $[2,2]$}] at (4,3) {};
                    \node[draw, minimum size=16, circle, fill=white, label=north: {\footnotesize $[1,1]$}] at (3,3) {};
                    
                    \node[draw, regular polygon, regular polygon sides=5, minimum size=16, fill=purple, label=north: {\footnotesize $[2,4]$}] at (1,4) {};
                    \node[draw, regular polygon, regular polygon sides=5, minimum size=16, fill=purple, label=north: {\footnotesize $[2,4]$}] at (5,4) {};
                    \node[draw, star, star points=6, fill=purple, label=north: {\footnotesize $[1,3]$}] at (2,4) {};
                    \node[draw, star, star points=6, fill=purple, label=north: {\footnotesize $[1,3]$}] at (4,4) {};
                    \node[draw, minimum size=16, fill=white, label=north: {\footnotesize $[0,2]$}] at (3,4) {};
                \end{tikzpicture}
                \caption{The product $C_5 \square C_5$ where the safe sets, safe vertices and safe houses for probe $B_1$ are indicated.
                Safe houses are indicated with the same colour and vertices that belong to the same safe set have the same shape and colour.
                The probed vertices are indicated as squares and empty vertices do not form part of a safe set.
                The distance from a vertex to $B_1$ is indicated above the vertex.}
                \label{fig: OxO B1}
            \end{figure}
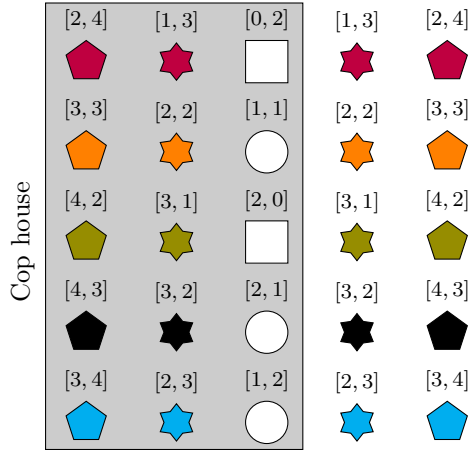
            
            If the Robber was at a vertex in column $2$, the Cop wins immediately.
            If not, the Robber is localized to the robber set $O_1 = \{v_{i,j}, v_{4-i,j} \}$ such that $N[O_1]$ is contained in rows $j-1$, $j$ and $j+1$.
            For the second probe, the Cop probes $B_2 = \{v_{4,j+1}, v_{2,j+1} \}$ such that $B_2$ is a translation of $B_1$, rotated by 90 degrees.
            Thus probe $B_2$ creates a cop house in rows $j-1$ to $j+1$.
            Since $N[O_1]$ is contained in these rows, the Cop wins.
        \end{proof}
        
        The proof for \cref{prop: loc(C5xC5)} is modified slightly for $C_5 \square C_3$ by changing the first probe to $B_1' = \{v_{2,1},v_{2,2} \}$ and keeping the second probe the same.
        \begin{proposition}
            Let $C_5 \square C_3$ be the product of two cycles.
            Then $\zeta(C_5 \square C_3)=2$.
        \end{proposition}
        It follows that for $m \text{ and }n$ odd, $\zeta(C_m \square C_n) = 2$, unless $m=n=3$.
         
    \subsection{Odd by Even}
        Next, consider the case where $m$ is odd and $n$ is even.
        Since $\psi(C_m)=2$, it follows that $\zeta(C_m \square C_n)=2$ for $n \geq 8$. To determine $\zeta(C_m \square C_n)$ for $n \leq 6$ we start by determining the safe houses for the chosen probes.
        \begin{lemma} \label{lem: O extra}
            Let $C_{2p+1} \square C_{2q}$ be a product of cycles with $p \geq 1$ and $q \in \{2,3 \}$.
            If the Cop probes $B_1 = \{v_{p,2q-1}, v_{p,q-1} \}$ in the first turn, all safe sets will be of the form $O = \{v_{i,j}, v_{2p-i,j}, v_{i,2q-2-j}, v_{2p-i,2q-2-j} \}$ for $i = 0,1,\ldots,2p$ and $j=0,1,\ldots,2q-1$.
            Further for $R_1$ as the set of all vertices $v_{x,y}$ where $x=0,1,\ldots,p$ and $y = q-1,q,\ldots,2q-1$, $R_1$ is a cop house.
        \end{lemma}
        \begin{proof}
            Say the Cop probes
            \begin{equation} \label{eq: OxE B1}
                B_1 = \{v_{p,2q-1}, v_{p,q-1} \}
            \end{equation}
            such that the distance vector $\vec{D}(B_1, v_{i,j})$ is given by
            \begin{equation} \label{eq: OxE dist eq}
                \vec{D}(B_1, v_{i,j}) = [|p-i|-|q-1-j|+q, |p-i|+|q-1-j|]
            \end{equation}
            for any vertex $v_{i,j}$.
            This is illustrated on $C_7 \square C_6$ in \cref{fig: OxE general}.
            \begin{figure}[h]
                \centering
                \begin{tikzpicture}[scale=1.2]
                    \draw[fill=black!20] (-0.3,1.7) rectangle (3.3,5.65);
                    \node at (-0.6,3.5) {\footnotesize $R_1$};
                
                    \node[draw, fill=purple, star, star points=4, label=north: {\footnotesize $[4,5]$}] at (0,0) {};
                    \node[draw, fill=purple, star, star points=4, label=north: {\footnotesize $[4,5]$}] at (6,0) {};
                    \node[draw, fill=purple, star, star points=4, label=north: {\footnotesize $[4,5]$}] at (6,4) {};
                    \node[draw, fill=purple, star, star points=4, label=north: {\footnotesize $[4,5]$}] at (0,4) {};
                    
                    \node[draw, fill=purple, regular polygon, regular polygon sides=5, minimum size=16, label=north: {\footnotesize $[3,4]$}] at (1,0) {};
                    \node[draw, fill=purple, regular polygon, regular polygon sides=5, minimum size=16, label=north: {\footnotesize $[3,4]$}] at (5,0) {};
                    \node[draw, fill=purple, regular polygon, regular polygon sides=5, minimum size=16, label=north: {\footnotesize $[3,4]$}] at (5,4) {};
                    \node[draw, fill=purple, regular polygon, regular polygon sides=5, minimum size=16, label=north: {\footnotesize $[3,4]$}] at (1,4) {};
                    
                    \node[draw, fill=purple, star, star points=6, label=north: {\footnotesize $[2,3]$}] at (2,0) {};
                    \node[draw, fill=purple, star, star points=6, label=north: {\footnotesize $[2,3]$}] at (2,4) {};
                    \node[draw, fill=purple, star, star points=6, label=north: {\footnotesize $[2,3]$}] at (4,0) {};
                    \node[draw, fill=purple, star, star points=6, label=north: {\footnotesize $[2,3]$}] at (4,4) {};
                    
                    \node[draw, fill=purple, regular polygon, regular polygon sides=7, minimum size=16, label=north: {\footnotesize $[1,2]$}] at (3,0) {};
                    \node[draw, fill=purple, regular polygon, regular polygon sides=7, minimum size=16, label=north: {\footnotesize $[1,2]$}] at (3,4) {};
                    
                    \node[draw, fill=orange, star, star points=4, label=north: {\footnotesize $[5,4]$}] at (0,1) {};
                    \node[draw, fill=orange, star, star points=4, label=north: {\footnotesize $[5,4]$}] at (0,3) {};
                    \node[draw, fill=orange, star, star points=4, label=north: {\footnotesize $[5,4]$}] at (6,1) {};
                    \node[draw, fill=orange, star, star points=4, label=north: {\footnotesize $[5,4]$}] at (6,3) {};
                    
                    \node[draw, fill=orange, regular polygon, regular polygon sides=5, minimum size=16, label=north: {\footnotesize $[4,3]$}] at (1,1) {};
                    \node[draw, fill=orange, regular polygon, regular polygon sides=5, minimum size=16, label=north: {\footnotesize $[4,3]$}] at (1,3) {};
                    \node[draw, fill=orange, regular polygon, regular polygon sides=5, minimum size=16, label=north: {\footnotesize $[4,3]$}] at (5,1) {};
                    \node[draw, fill=orange, regular polygon, regular polygon sides=5, minimum size=16, label=north: {\footnotesize $[4,3]$}] at (5,3) {};
                    
                    \node[draw, fill=orange, star, star points=6, label=north: {\footnotesize $[3,2]$}] at (2,1) {};
                    \node[draw, fill=orange, star, star points=6, label=north: {\footnotesize $[3,2]$}] at (2,3) {};
                    \node[draw, fill=orange, star, star points=6, label=north: {\footnotesize $[3,2]$}] at (4,1) {};
                    \node[draw, fill=orange, star, star points=6, label=north: {\footnotesize $[3,2]$}] at (4,3) {};
                    
                    \node[draw, fill=orange, regular polygon, regular polygon sides=7, minimum size=16, label=north: {\footnotesize $[2,1]$}] at (3,1) {};
                    \node[draw, fill=orange, regular polygon, regular polygon sides=7, minimum size=16, label=north: {\footnotesize $[2,1]$}] at (3,3) {};
                    
                    \node[draw, fill=cyan, star, star points=4, label=north: {\footnotesize $[6,3]$}] at (0,2) {};
                    \node[draw, fill=cyan, star, star points=4, label=north: {\footnotesize $[6,3]$}] at (6,2) {};
                    
                    \node[draw, fill=cyan, regular polygon, regular polygon sides=5, minimum size=16, label=north: {\footnotesize $[5,2]$}] at (1,2) {};
                    \node[draw, fill=cyan, regular polygon, regular polygon sides=5, minimum size=16, label=north: {\footnotesize $[5,2]$}] at (5,2) {};
                    
                    \node[draw, fill=cyan, star, star points=6, label=north: {\footnotesize $[4,1]$}] at (2,2) {};
                    \node[draw, fill=cyan, star, star points=6, label=north: {\footnotesize $[4,1]$}] at (4,2) {};
                    
                    \node[draw, fill=white, minimum size=16, label=north: {\footnotesize $[3,0]$}] at (3,2) {};
                    \node[draw, fill=white, minimum size=16, label=north: {\footnotesize $[0,3]$}] at (3,5) {};
                    
                    \node[draw, fill=lime, star, star points=4, label=north: {\footnotesize $[3,6]$}] at (0,5) {};
                    \node[draw, fill=lime, star, star points=4, label=north: {\footnotesize $[3,6]$}] at (6,5) {};
                    
                    \node[draw, fill=lime, regular polygon, regular polygon sides=5, minimum size=16, label=north: {\footnotesize $[2,5]$}] at (1,5) {};
                    \node[draw, fill=lime, regular polygon, regular polygon sides=5, minimum size=16, label=north: {\footnotesize $[2,5]$}] at (5,5) {};
                    
                    \node[draw, fill=lime, star, star points=6, label=north: {\footnotesize $[1,4]$}] at (2,5) {};
                    \node[draw, fill=lime, star, star points=6, label=north: {\footnotesize $[1,4]$}] at (4,5) {};
                \end{tikzpicture}
                \caption{The graph $C_7 \square C_6$ with probe $B_1$ as in Equation \cref{eq: OxE B1}.
                The distances from vertices to $B_1$ as well as the cop house $R_1$ are shown.}
                \label{fig: OxE general}
            \end{figure}
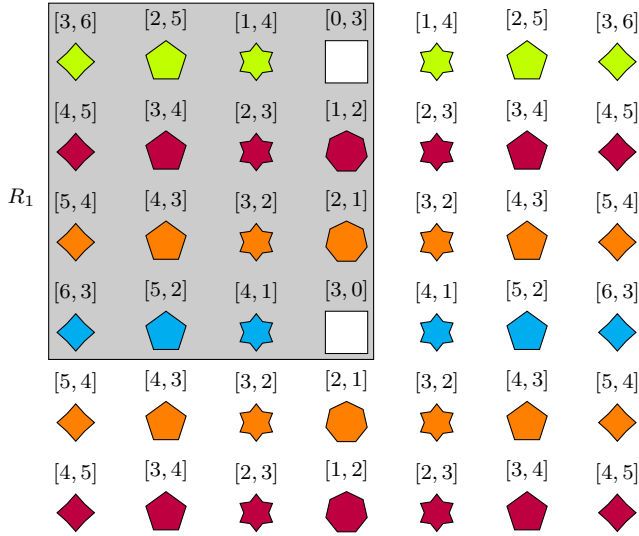
            
            The second difference is given by $DD(B_1,v_{i,j}) = 2|q-1-j|-q$ and is therefore not dependant on the column of vertex $v_{i,j}$ and only on its row.
            Consider two vertices $v_{i_1,j_1}, v_{i_2,j_2}$ and let $DD(B_1,v_{i_1,j_1}) = DD(B_1,v_{i_2,j_2})$ such that
            \begin{equation} \label{eq: OxE SH solve}
                |q-1-j_1| = |q-1-j_2|. 
            \end{equation}
            There are two solutions to Equation \cref{eq: OxE SH solve}: $j_1=j_2$ and $j_1+j_2 = 2q-2$.
            All vertices in the same row are therefore in the same safe house, where vertices in different rows are in the same safe house only if $j_1+j_2 = 2q-2$.
            It follows that that every safe house contains two rows of $C_{2p+1} \square C_{2q}$, except the two safe houses containing rows $q-1$ and $2q-1$ respectively.
            In order to calculate the safe sets, let $\vec{D}(B_1,v_{i_1,j_1}) = \vec{D}(B_1,v_{i_2,j_2})$.
            Since the two vertices are in the same safe house, it follows from Equations \cref{eq: OxE dist eq} and \cref{eq: OxE SH solve} that two vertices are in the same safe set if $|p-i_1| = |p-i_2|$.
            The only nontrivial solution is $i_1+i_2=2p$.
            Therefore all safe sets have the form $O_1 = \{v_{i,j}, v_{2p-i,j}, v_{i,2q-2-j}, v_{2p-i,2q-2-j} \}$ for $i = 0,1,\ldots,2p$ and $j=0,1,\ldots,2q-1$.
            Note that if $j=q-1$, $j=2q-1$ or $i=p$, the safe set only contains two vertices.
            
            Next consider $R_1$.
            By the solution to Equation \eqref{eq: OxE SH solve}, two vertices in $R_1$ only belong to the same safe house if they are in the same row.
            Therefore two vertices $v_{x_1,y_1}, v_{x_2,y_2}$ in $R_1$ are only part of the same safe set if $|p-x_1| = |p-x_2|$ such that $x_1+x_2=2p$.
            This is never true inside $R_1$ and therefore every two vertices in $R_1$ belong to different safe sets and $R_1$ is a cop house.
        \end{proof}
        Since the $C_m \square C_n$ is vertex transitive, the following corollary follows:
        \begin{corollary} \label{cor: R2 cop house}
            Say the Cop probes $B_2 = \{v_{a+p,b+q}, v_{a+p,b} \}$ in the second turn such that $B_2 = g(B_1)$, where $g$ is a translation.
            Then for $R_2 = g(R_1)$, $R_2$ is a cop house.
            Further two distinct vertices $v_{i_1,j_1}$ and $v_{i_2,j_2}$ are only part of the same safe set if $(i_1+i_2) \equiv 2a+2p\bmod{(2p+1)}$ or $(j_1+j_2) \equiv 2b\bmod{(2q)}$.
            Note that if $i_1 \neq i_2$ and $j_1 \neq j_2$, then both these equations need to hold.
        \end{corollary}
        The lemma can also easily be adapted for the even by even case:
        \begin{corollary} \label{cor: ExE extra}
            Let $C_{2p} \square C_{2q}$ be a product of cycles with $p \geq  q \geq 4$.
            If the Cop probes $B_1 = \{v_{p,2q-1}, v_{p,q-1} \}$ in the first turn, all safe sets will be of the form $O = \{v_{i,j}, v_{2p-i,j}, v_{i,2q-2-j}, v_{2p-i,2q-2-j} \}$ for $i = 0,1,\ldots,2p$ and $j=0,1,\ldots,2q-1$.
            Further, for $R_1$ defined as in \cref{lem: O extra}, $R_1$ is a cop house.
            Also if $f$ is a translation such that $B_2 = f(B_1)$, then $R_2 = f(R_1)$ is a cop house.
        \end{corollary}
        
        Now for the case of $m$ odd and $n \leq 6$ we omit the restriction that $m \geq n$.
        If this restriction is included, a separate proof will be needed for even by odd, which will be equivalent to the one given here.
        \begin{proposition} \label{prop: OxE small case}
            Let $C_{2p+1} \square C_{2q}$ be a product of cycles with $p \geq 1$ where $q \in \{2,3 \}$.
            Then $\zeta(C_{2p+1} \square C_{2q}) = 2$.
        \end{proposition}
        \begin{proof}
            In the first turn, the Cop probes $B_1 = \{v_{p,2q-1}, v_{p,p-1} \}$ such that the Robber is localized to robber set
            \begin{equation} \label{eq: O}
                O_1 = \{v_{i,j}, v_{2p-i,j}, v_{i,2q-2-j}, v_{2p-i,2q-2-j} \}
            \end{equation}
            by \cref{lem: O extra}, where $i = 0,1,\ldots,2p$ and $j=0,1,\ldots,2q-1$.
            Now define $d_i = d(v_{i,j}, v_{2p-i,j})$ and $d_j = d(v_{i,j}, v_{i,2q-2-j})$ such that the robber set can be written as $O_1 = \{v_{a,b}, v_{a+d_i,b}, v_{a,b+d_j}, v_{a+d_i,b+d_j} \}$ where $v_{i,j}=v_{a,b}$ need not be true.
            The distances $d_i$ and $d_j$ are given by $d_i = \min\{2i+1,2p-2i \}$ and $d_j = \min\{2j+2,2q-2-2j \}$.
            Note that $d_i$ and $d_j$ are also calculated modulo $m$ and $n$ respectively.
            It follows that $d_i\leq p$ and since $q \in \{2,3\}$, we have that $d_j \in \{0,2 \}$.
            In the second turn, the Cop probes
            \begin{equation} \label{eq: B2 OxE}
                B_2 = \{v_{a+p,b+q}, v_{a+p,b} \}
            \end{equation}
            such that $B_2 = g(B_1)$ where $g$ is some translation function.
            The vertices of $O_1$ will be labeled $v_{a,b}=u_1,v_{a,b+d_j}=u_2, v_{a+d_i,b+d_j}=u_3$ and $v_{a+d_i,b}=u_4$ with neighbours $u_l^N, u_l^E, u_l^S, u_l^W$ for $l=1,2,3,4$.
            Let $R_2$ be the set of all vertices $v_{s,t}$ where $s = a,a+1,\ldots,a+p$ and $t = b,b+1,\ldots,b+q$ such that $R_2 = g(R_1)$.
            Then $R_2$ is a cop house by \cref{cor: R2 cop house}.
            The vertices in $N[O_1]$ as well as probe $B_2$ are illustrated in \cref{fig: OxE N[O]}.
            In the figure, the region $R_2$ is indicated with a dotted square.
            \begin{figure}[h]
                \centering
                \begin{tikzpicture}[scale=1]
                    \draw[dotted] (0,0) rectangle (7,6.2);
                
                    \node[draw, circle, label=south: {\footnotesize $u_1$}, fill=red] at (0,0) {};
                    \node[draw, label=south: {\footnotesize $v_{a+p,b}$}, fill=white] at (7,0) {};
                    \node[draw, label=south: {\footnotesize $v_{a+p,b+q}$}, fill=white] at (7,6.2) {};
                    
                    \node[draw, circle, label=south: {\footnotesize $u_1^W$}, fill=red!50] at (-1,0) {};
                    \node[draw, circle, label=south: {\footnotesize $u_1^E$}, fill=red!50] at (1,0) {};
                    \node[draw, circle, label=south: {\footnotesize $u_1^S$}, fill=red!50] at (0,-1) {};
                    \node[draw, circle, label=south: {\footnotesize $u_1^N$}, fill=red!50] at (0,1) {};
                    
                    \node[draw, circle, label=south: {\footnotesize $u_4$}, fill=red] at (4.5,0) {};
                    \node[draw, circle, label=south: {\footnotesize $u_3$}, fill=red] at (4.5,4.4) {};
                    \node[draw, circle, label=south: {\footnotesize $u_2$}, fill=red] at (0,4.4) {};
                    
                    \node[draw, circle, fill=red!50, label=south: {\footnotesize $u_4^W$}] at (3.5,0) {};
                    \node[draw, circle, fill=red!50, label=south: {\footnotesize $u_4^E$}] at (5.5,0) {};
                    \node[draw, circle, fill=red!50, label=south: {\footnotesize $u_4^S$}] at (4.5,-1) {};
                    \node[draw, circle, fill=red!50, label=south: {\footnotesize $u_4^N$}] at (4.5,1) {};
                    
                    \node[draw, circle, label=south: {\footnotesize $u_2^S$}, fill=red!50] at (0,3.4) {};
                    \node[draw, circle, label=south: {\footnotesize $u_2^N$}, fill=red!50] at (0,5.4) {};
                    \node[draw, circle, label=south: {\footnotesize $u_2^E$}, fill=red!50] at (1,4.4) {};
                    \node[draw, circle, label=south: {\footnotesize $u_2^W$}, fill=red!50] at (-1,4.4) {};
                    
                    \node[draw, circle, label=south: {\footnotesize $u_3^E$}, fill=red!50] at (5.5,4.4) {};
                    \node[draw, circle, label=south: {\footnotesize $u_3^W$}, fill=red!50] at (3.5,4.4) {};
                    \node[draw, circle, label=south: {\footnotesize $u_3^N$}, fill=red!50] at (4.5,5.4) {};
                    \node[draw, circle, label=south: {\footnotesize $u_3^S$}, fill=red!50] at (4.5,3.4) {};
                \end{tikzpicture}
                \caption{The vertices in $N[O]$ where the two probed vertices are squares and the cop house $R_2$ is indicated with a dotted square.
                Note that $v_{a,b}=u_1,v_{a,b+d_j}=u_2, v_{a+d_i,b+d_j}=u_3$ and $v_{a+d_i,b}=u_4$.}
                \label{fig: OxE N[O]}
            \end{figure}
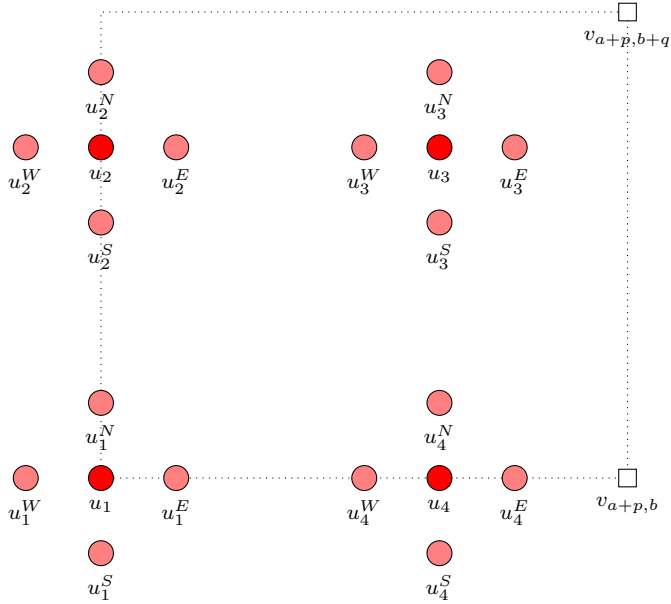
            
            We now show that every safe set in the second turn is a vertical or horizontal safe pair, where the two vertices in the safe set are at distance one or two from each other.
            First, consider $N[\{u_1, u_4\}]$.
            By \cref{cor: R2 cop house} all vertices in row $b$ belong to a safe house and no vertices outside of this row are part of the same safe house.
            Further, safe sets in this row only contain two vertices and therefore none of the vertices in row $b$ are part of vertical safe pairs.
            By \cref{cor: R2 cop house} it follows that $\{u_1,u_1^W \}, \{u_1^N, u_1^S \}$ and $\{u_4^N, u_4^S \}$ are safe sets.
            Vertex $u_1^E$ can only be in a safe set with a vertex outside $R_2$ in row $b$.
            Thus the only option is $u_4^E$ if $d_i=p$.
            Then, $i_1+i_2= (a+1) + (a+d_i+1) = 2a+p+2$.
            This only satisfies $(i_1+i_2) \equiv 2a+2p\bmod{(2p+1)}$ if $p=2$, in which case the safe set $\{u_1^E, u_4^E \} = S_2^h$.
            If $d_i=p>2$, then $\{u_4^E, u_4^W\}$ is a horizontal safe pair at distance two. 
            Otherwise the vertices $u_4^W, u_4, u_4^E$  will be inside cop house $R_2$ and are therefore not part of the same safe set.
            Since every safe pair has a vertex inside the cop house, no pair is part of a safe set containing 4 vertices.
            
            If $d_i\leq p-2$, then $u_2^W$ is the only neighbour of $u_2$ and $u_3$ outside of $R_2$.
            Otherwise a similar argument can be applied to the vertices in $N[\{u_2, u_3\}]$.
            Therefore all safe sets in $N[\{O_1\}]$ contain two vertices in the same row or column, a distance of one or two apart.
            
            In the next turn, the Cop chooses $B_3$ such that the Robber is localized to a diagonal safe pair.
            This probe will depend on the robber set $O_2$, where $O_2 = \{v_{x,y}, v_{x+1,y} \}$, $O_2 = \{v_{x,y}, v_{x+2,y} \}$ or $O_2 = \{v_{x,y}, v_{x,y-2} \}$.
            \begin{case}[$O_2 = \{v_{x,y}, v_{x+1,y} \}$] \label{case: OxE d=1}
                The Cop probes $B_3 = \{v_{x-1,y+1}, v_{x,y} \}$ such that the distances from vertices in $N[O_2]$ are given in \cref{tab: OxE case d=1}.
                It can be seen that the only safe sets, are diagonal safe pairs.
                \begin{table}[h]
                    \centering
                    \setlength{\tabcolsep}{4.2pt} 
                    \caption{The distances from vertices in $N[O_2]$ to $B_3$ for \cref{case: OxE d=1}.
                    Note that the only safe sets, are diagonal safe pairs.}
                    \begin{tabular}{c | c c c c c c c c}
                        $v \in N[O_2]$ & $v_{x,y}$ & $v_{x-1,y}$ & $v_{x+1,y}$ & $v_{x,y-1}$ & $v_{x,y+1}$ & $v_{x+1,y+1}$ & $v_{x+1,y-1}$ & $v_{x+2,y}$ \\ \hline
                        $\vec{D}(B_3, v)$, $p=1$ & $[2,0]$ & $[1,1]$ & $[2,1]$ & $[3,1]$ & $[1,1]$ & $[1,2]$ & $[3,2]$ & N/A \\
                        $\vec{D}(B_3, v)$, $p=2$ & $[2,0]$ & $[1,1]$ & $[3,1]$ & $[3,1]$ & $[1,1]$ & $[2,2]$ & $[4,2]$ & $[3,2]$ \\
                        $\vec{D}(B_3, v)$, $p \geq 3$ & $[2,0]$ & $[1,1]$ & $[3,1]$ & $[3,1]$ & $[1,1]$ & $[2,2]$ & $[4,2]$ & $[4,2]$
                    \end{tabular}
                    \label{tab: OxE case d=1}
                \end{table}
            \end{case}
            
            \begin{case}[$O_2 = \{v_{x,y}, v_{x+2,y} \}$] \label{case: OxE hord=2}
                Note that this case only holds for $p \geq 2$.
                The Cop probes $B_3 = \{v_{x+1,y+1}, v_{x,y} \}$, where the distances from vertices in $N[O_2]$ to $B_3$ are given in \cref{fig: OxE case hord=2}.
                It can again be seen that the only safe sets, are diagonal safe pairs.
                Note that if $p=2$, then $\vec{D}(B_3, v_{x+2,y}) = [3,2]$ and not $[3,3]$.
                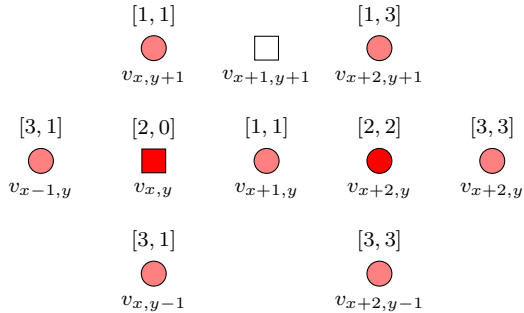
\begin{figure}[h]
                    \centering
                    \begin{tikzpicture}[scale=1.5]
                        \node[draw, minimum size=9, fill=red, label=south: {\footnotesize $v_{x,y}$}, label=north: {\footnotesize $[2,0]$}] at (0,0) {};
                        \node[draw, fill=red, circle, label=south: {\footnotesize $v_{x+2,y}$}, label=north: {\footnotesize $[2,2]$}] at (2,0) {};
                        \node[draw, minimum size=9, fill=white, label=south: {\footnotesize $v_{x+1,y+1}$}] at (1,1) {};
                        
                        \node[draw, circle, fill=red!50, label=south: {\footnotesize $v_{x,y-1}$}, label=north: {\footnotesize $[3,1]$}] at (0,-1) {};
                        \node[draw, circle, fill=red!50, label=south: {\footnotesize $v_{x,y+1}$}, label=north: {\footnotesize $[1,1]$}] at (0,1) {};
                        \node[draw, circle, fill=red!50, label=south: {\footnotesize $v_{x+1,y}$}, label=north: {\footnotesize $[1,1]$}] at (1,0) {};
                        \node[draw, circle, fill=red!50, label=south: {\footnotesize $v_{x-1,y}$}, label=north: {\footnotesize $[3,1]$}] at (-1,0) {};
                        \node[draw, circle, fill=red!50, label=south: {\footnotesize $v_{x+2,y}$}, label=north: {\footnotesize $[3,3]$}] at (3,0) {};
                        \node[draw, circle, fill=red!50, label=south: {\footnotesize $v_{x+2,y+1}$}, label=north: {\footnotesize $[1,3]$}] at (2,1) {};
                        \node[draw, circle, fill=red!50, label=south: {\footnotesize $v_{x+2,y-1}$}, label=north: {\footnotesize $[3,3]$}] at (2,-1) {};
                    \end{tikzpicture}
                    \caption{An illustration of probe $B_3$ in \cref{case: OxE hord=2}.
                    It can again be seen that the only safe sets, are diagonal safe pairs.
                    Note that if $p=2$, then $\vec{D}(B_3, v_{x+2,y}) = [3,2]$ and not $[3,3]$.}
                    \label{fig: OxE case hord=2}
                \end{figure}
            \end{case}
            
            \begin{case}[$O_2 = \{v_{x,y}, v_{x,y-2} \}$] \label{case: OxE verd=2}
                The Cop now probes $B_3 = \{v_{x,y}, v_{x-1,y+1} \}$.
                The distances from vertices in $N[O_2]$ to $B_3$ are given in \cref{tab: OxE verd=2} for $p \geq 2$.
                These distances are given in \cref{tab: OxE verd=2 p=1} for $p=1$.
                In both tables it can be seen that the only safe sets are diagonal safe pairs.
                \begin{table}[h]
                    \centering
                    \setlength{\tabcolsep}{2.4pt} 
                    \caption{The distances from vertices in $N[O_2]$ to $B_3$ for \cref{case: OxE verd=2} when $p \geq 2$.
                    Note that the only safe sets, are diagonal safe pairs.}
                    \begin{tabular}{c | c c c c c c c c c}
                        $v \in N[O_2]$ & $v_{x,y}$ & $v_{x-1,y}$ & $v_{x+1,y}$ & $v_{x,y+1}$ & $v_{x,y-1}$ & $v_{x,y-2}$ & $v_{x-1,y-2}$ & $v_{x+1,y-2}$ & $v_{x,y-3}$ \\ \hline
                        $\vec{D}(B_3, v)$, $q=2$ & $[0,2]$ & $[1,1]$ & $[1,3]$ & $[1,1]$ & $[1,3]$ & $[2,2]$ & $[3,1]$ & $[3,3]$ & N/A \\
                        $\vec{D}(B_3, v)$, $q=3$ & $[0,2]$ & $[1,1]$ & $[1,3]$ & $[1,1]$ & $[1,3]$ & $[2,4]$ & $[3,3]$ & $[3,5]$ & $[3,3]$
                    \end{tabular}
                    \label{tab: OxE verd=2}
                \end{table}
                \begin{table}[h!]
                    \centering
                    \setlength{\tabcolsep}{2.4pt} 
                    \caption{The distances from vertices in $N[O_2]$ to $B_3$ for \cref{case: OxE verd=2} when $p = 1$.
                    Note that the only safe sets, are diagonal safe pairs.}
                    \begin{tabular}{c | c c c c c c c c c}
                        $v \in N[O_2]$ & $v_{x,y}$ & $v_{x-1,y}$ & $v_{x+1,y}$ & $v_{x,y+1}$ & $v_{x,y-1}$ & $v_{x,y-2}$ & $v_{x-1,y-2}$ & $v_{x+1,y-2}$ & $v_{x,y-3}$ \\ \hline
                        $\vec{D}(B_3, v)$, $q=2$ & $[0,2]$ & $[1,1]$ & $[1,2]$ & $[1,1]$ & $[1,3]$ & $[2,2]$ & $[3,1]$ & $[3,2]$ & N/A \\
                        $\vec{D}(B_3, v)$, $q=3$ & $[0,2]$ & $[1,1]$ & $[1,2]$ & $[1,1]$ & $[1,3]$ & $[2,4]$ & $[3,3]$ & $[3,4]$ & $[3,3]$
                    \end{tabular}
                    \label{tab: OxE verd=2 p=1}
                \end{table}
            \end{case}
            
            Now say the Robber is localized to a set $O_3 = \{v_{x,y}, v_{x+1,y+1} \}$.
            If $p \geq 2$, the Cop probes $B_4 = \{v_{x-p+1,y}, v_{x-p,y-1} \}$ such that the distances from vertices in $N[O_3]$ to $B_4$ are given in \cref{fig:OxE small case}.
            \begin{figure}[h]
                \centering
                \begin{tikzpicture}[scale=2]
                    \node[draw, circle, fill=red, label=south: {\footnotesize $v_{x,y}$}, label=north: {\footnotesize $[p-1,p+1]$}] at (0,0) {};
                    \node[draw, circle, fill=red, label=south: {\footnotesize $v_{x+1,y+1}$}, label=north: {\footnotesize $[p+1,p+2]$}] at (1,0.9) {};
                    
                    \node[draw, circle, fill=red!50, label=south: {\footnotesize $v_{x-1,y}$}, label=north: {\footnotesize $[p-2,p]$}] at (-1,0) {};
                    \node[draw, circle, fill=red!50, label=south: {\footnotesize $v_{x+1,y}$}, label=north: {\footnotesize $[p,p+1]$}] at (1,0) {};
                    \node[draw, circle, fill=red!50, label=south: {\footnotesize $v_{x,y-1}$}, label=north: {\footnotesize $[p,p]$}] at (0,-0.9) {};
                    \node[draw, circle, fill=red!50, label=south: {\footnotesize $v_{x,y+1}$}, label=north: {\footnotesize $[p,p+2]$}] at (0,0.9) {};
                    
                    \node[draw, circle, fill=red!50, label=south: {\footnotesize $v_{x+2,y+1}$}, label=north: {\footnotesize $[p+1,p+1]$}] at (2,0.9) {};
                    \node[draw, circle, fill=red!50, label=south: {\footnotesize $v_{x+1,y+2}$}, label=north: {\footnotesize $[p+2,p+3]$}] at (1,1.8) {};
                \end{tikzpicture}
                \caption{The distances from vertices in $N[O_3]$ to $B_4$ for $p \geq 2$ as in the proof of \cref{prop: OxE small case}. 
                It can be seen that all vertices in $N[O_3]$ are resolved by $B_4$ and hence the Cop wins.}
                \label{fig:OxE small case}
            \end{figure}
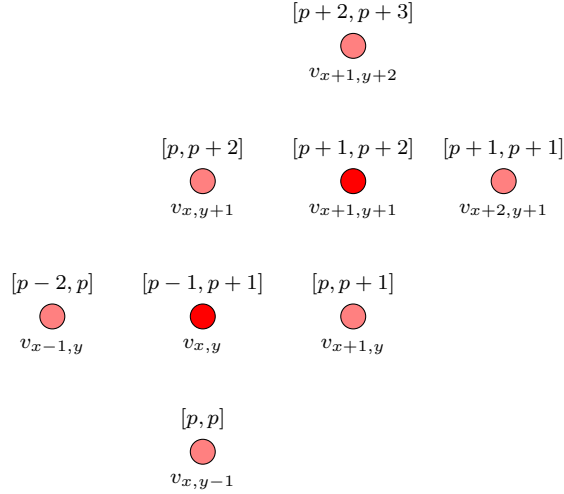
            
            If $p=1$, the Cop probes $B_4 = \{v_{x-1,y}, v_{x,y-1} \}$ such that the distances from vertices in $N[O_3]$ to $B_4$ are given in \cref{tab:OxE small case p=1}.
            Note that if $q=2$, then $\vec{D}(B_4,v_{x+1,y+2})=[2,4]$ at not $[3,4]$.
            \begin{table}[h]
                \centering
                \caption{The distances from vertices in $N[O_3]$ to $B_4$ for $p = 1$ as in the proof of \cref{prop: OxE small case}. 
                Note that if $q=2$, then $\vec{D}(B_4,v_{x+1,y+2})=[2,4]$ at not $[3,4]$.}
                \resizebox{\textwidth}{!}{%
                \begin{tabular}{c|cccccccc}
                    $v \in N[O_3]$ & $v_{x,y}$ & $v_{x+1,y+1}$ & $v_{x,y-1}$ & $v_{x-1,y}$ & $v_{x+1,y}$ & $v_{x-1,y+1}$ & $v_{x,y+1}$ & $v_{x+1,y+2}$ \\ 
                    $\vec{D}(B_4,v)$ & $[1,1]$ & $[2,3]$ & $[2,0]$ & $[0,2]$ & $[1,2]$ & $[1,3]$ & $[2,2]$ & $[3,4]$
                \end{tabular}
                }
                \label{tab:OxE small case p=1}
            \end{table}
            
            All vertices in $N[O_3]$ are uniquely defined by their distance to $B_4$ and hence the Cop wins.
            If the Robber was localized to $O_3 = \{v_{x,y}, v_{x+1,y-1} \}$, the Cop probes $B_4 = \{v_{x-p+1,y}, v_{x-p,y+1} \}$ if $p \geq 2$ and $B_4 = \{v_{x-1,y}, v_{x,y+1} \}$ if $p=1$ such that results follow similarly.
        \end{proof}
        
    \subsection{Even by Even}
        For even $m$ and $n$, first consider the case where $m \geq n \geq 8$:
        \begin{proposition} \label{ExE main}
            Let $C_{2p} \square C_{2q}$ be a product of cycles with $p,q \geq 4$ and $p \geq q$.
            Then $\zeta(C_{2p} \square C_{2q}) = 2$.
        \end{proposition}
        \begin{proof}
            In the first turn, the Cop probes $B_1 = \{v_{p,2q-1}, v_{p,q-1} \}$ such that the Robber is localized to $O_1 = \{v_{i,j}, v_{2p-i,j}, v_{i,2q-2-j}, v_{2p-i,2q-2-j} \}$ for $i \in \{ 0,1,\ldots,2p\}$ and $j \in \{0,1,\ldots,2q-1\}$ by \cref{cor: ExE extra}.
            The Cop's second probe depends on $d_i$ and $d_j$, where $d_i = d(v_{i,j}, v_{2p-i,j})$ and $d_j = d(v_{i,j}, v_{i,2q-2-j})$.
            These two distances are calculated as follows: $d_j = \min\{2j+2,2q-2-2j \}$ as before and $d_i = \min\{2i,2p-2i \}$.
            The Robber set is again given by $O_1 = \{v_{a,b}, v_{a+d_i,b}, v_{a,b+d_j}, v_{a+d_i,b+d_j} \}$.
            \begin{strat}[$\mathbf{d_i \leq p-2}$ and $\mathbf{d_j \leq q-2}$] \label{strat: OxE short}
                The Cop probes \\ $B_2 = \{v_{a-1+p, b-1+q}, v_{a-1+p, b-1} \}$ such that $B_2$ is a translation of $B_1$.
                Let $f$ be a translation such that $B_2 = f(B_1)$ and let $R_2$ be the set of all vertices $v_{w,z}$ where $w = a-1,a,\ldots,a-1+p$ and $z = b-1,b,\ldots,b-1+q$.
                Then $R_2 = f(R_1)$ such that it is a cop house by \cref{cor: ExE extra}.
                Since $a+d_i \leq a+p-2$ and $b+d_j \leq b+q-2$, the neighbourhood $N[O_1]$ is contained in $R_2$ and therefore the Cop wins.
            \end{strat}
            
            \begin{strat}[$\mathbf{d_i > p-2}$ or $\mathbf{d_j > q-2}$] \label{strat: OxE long}
                This means that at least one of the following holds: $d_i \in \{p-1,p \}$ or $d_j \in \{q-1,q \}$.
                Note that in the proof of \cref{prop: OxE small case} we have that $d_j > q-2$ and therefore a similar strategy can be used here.
                The Cop now probes $B_2' = \{v_{a+p,b+q}, v_{a+p,b} \}$ as in Equation \cref{eq: B2 OxE} such that we again have that each safe set is either a horizontal or vertical safe pair.
            \end{strat}
            Notice that two vertices $v_{i_1,j_1}, v_{i_2,j_2}$ are in the same safe set if and only if $i_1+i_2 \equiv 2x\bmod{(2p)}$.
            Therefore the same argument as in the proof of \cref{prop: OxE small case} can be used to show that if vertices $v_{i_1,j_1}$ and $v_{i_2,j_2}$ are not part of the same neighbourhood $N[\{u_i \}]$, they are not part of the same safe set.
            Since every safe set is a vertical or diagonal pair of distance one or two, the Cop wins in the next turn by using \cref{strat: OxE short}.
        \end{proof}
         
        Now the consider the case where $n=6$:
        \begin{proposition} \label{prop: C2mxC6}
            Let $C_{2p} \square C_6$ be a product of cycles with $p \geq 3$.
            Then \\ $\zeta(C_{2p} \square C_6) = 2$.
        \end{proposition}
        \begin{proof}
            The Cop plays with two cops by using the imagined localization game on $C_{2p+1} \square C_6$.
            In the first turn, the Cop probes $B_1 = \{v_{p,5}, v_{p,2} \}$ as in the imagined game.
            A similar proof to \cref{lem: O extra} can be given here to show that all safe sets have the form $O_1 = \{v_{i,j}, v_{2p-i,j}, v_{i,4-j}, v_{2p-i,4-j} \}$ or equivalently $O_1 = \{v_{i,j}, v_{i+d_i,j}, v_{i,j+d_j}, v_{i+d_i,j+d_j} \}$ where $d_i = d(v_{i,j}, v_{2p-i,j})$ and $d_j = d(v_{i,j}, v_{i,4-j})$ for $i \in \{0,1,\ldots,2p-1 \}$ and $j \in \{0,1,\ldots,5 \}$.
            Therefore the safe sets in the real game are the same as in the imagined game with the exception that $i \in \{0,1,\ldots,2p \}$ in the imagined game.
            Therefore all robber sets in the real game are possible in the imagined game.
            For the second turn in the imagined game the Cop uses the strategy used in the second probe of the proof of \cref{prop: OxE small case}  to localize the Robber to a safe set containing only two vertices, a distance of one or two apart.
            Since $i\leq 2p$, the real game does not contain horizontal safe pairs at distance 1 and hence the Robber is localized to a robber set of the form $O_2 = \{v_{x,y}, v_{x+2,y} \}$ or $O_2 = \{v_{x,y}, v_{x,y-2} \}$.
            These two cases are possible in the imagined game and handled in \cref{case: OxE hord=2} and \cref{case: OxE verd=2} for the Cop's next probe.
            Thus the Robber is localized to a diagonal safe pair $O_3$.
            
            If $O_3 = \{v_{x,y}, v_{x-1,y-1} \}$, the imagination strategy is not used.
            The Cop probes $B_4 = \{v_{x+1,y+1}, v_{x+1,y-2} \}$.
            The explicit distances from the vertices in $N[O_3]$ to $B_4$ are given in \cref{fig: Even/Even C6 short strat}, where it can be seen that no two distances are the same.
            The index of a vertex is shown below the vertex and its distance to $B_4$ is shown above it.
            The vertices of $B_4$ are squares, the vertices in $O_3$ are darker red and the vertices in $N(O_3)$ in lighter red.
            \begin{figure}[h]
                \centering
                \begin{tikzpicture}[scale=1.5]
                    \node[draw, circle, label=south: {\footnotesize $v_{x,y}$}, fill=red, label=north: {\footnotesize $[2,3]$}] at (0,0) {};
                    \node[draw, circle, label=south: {\footnotesize $v_{x-1,y-1}$}, fill=red, label=north: {\footnotesize $[4,3]$}] at (-1,-1) {};
                    
                    \node[draw, minimum size=8.5, label=south: {\footnotesize $v_{x+1,y+1}$}, fill=gray] at (1,1) {};
                    \node[draw, minimum size=8.5, label=south: {\footnotesize $v_{x+1,y-2}$}, fill=gray] at (1,-2) {};
                    
                    \node[draw, circle, label=south: {\footnotesize $v_{x,y+1}$}, fill=red!50, label=north: {\footnotesize $[1,4]$}] at (0,1) {};
                    \node[draw, circle, label=south: {\footnotesize $v_{x,y-1}$}, fill=red!50, label=north: {\footnotesize $[3,2]$}] at (0,-1) {};
                    \node[draw, circle, label=south: {\footnotesize $v_{x+1,y}$}, fill=red!50, label=north: {\footnotesize $[1,2]$}] at (1,0) {};
                    \node[draw, circle, label=south: {\footnotesize $v_{x-1,y}$}, fill=red!50, label=north: {\footnotesize $[3,4]$}] at (-1,0) {};
                    
                    \node[draw, circle, label=south: {\footnotesize $v_{x-2,y-1}$}, fill=red!50, label=north: {\footnotesize $[5,4]$}] at (-2,-1) {};
                    \node[draw, circle, label=south: {\footnotesize $v_{x-1,y-2}$}, fill=red!50, label=north: {\footnotesize $[5,2]$}] at (-1,-2) {};
                    
                    \node[draw, circle, fill=gray] at (1,-1) {};
                    \node[draw, circle, fill=gray] at (0,-2) {};
                \end{tikzpicture}
                \caption{The neighbourhood $N[O_3]$ and probe $B_4$ as in the proof of \cref{prop: C2mxC6}.
                The index of a vertex is shown below the vertex and its distance to $B_4$ is shown above it.
                The vertices of $B_4$ are squares, the vertices in $O_3$ are darker red and the vertices in $N(O_3)$ in lighter red.
                Note that no two distances are the same.}
                \label{fig: Even/Even C6 short strat}
            \end{figure}
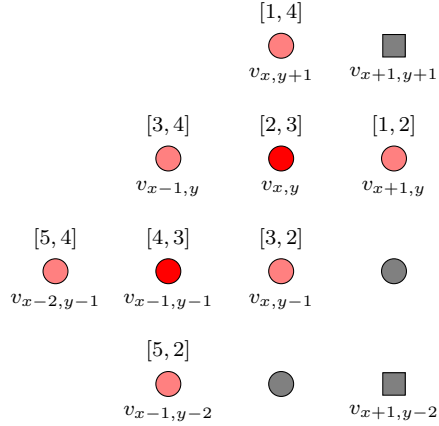
                    
            Note that if $O_3= \{v_{x,y}, v_{x+1,y-1} \}$, the Cop probes $B_4 = \{v_{x-1,y+1}, v_{x-1,y-2} \}$ and results follow similarly.
            Thus the Robber is located and the Cop wins.
        \end{proof}
        
        In order to calculate the localization number of $C_{2p} \square C_4$, the following lemmas are used:
        \begin{lemma}[\cite{Newlocbounds}] \label{lem: Bipart dist}
            Let $G$ be a bipartite graph, where $v \in V(G)$ and $w \in N(v)$. 
            Say the Cop probes $B= \{b_1, b_2, \ldots, \}$ in some turn and let $d_i=d(b_i, v)$.
            Then $d(b_i,w) \in \{d_i-1, d_i+1 \}$.
        \end{lemma}
        \begin{lemma}[\cite{Dimensions}] \label{lem: dim(CxC)}
            Let $C_m \square C_n$ be the product of cycles where $m,n \geq 3$.
            Then
            \begin{equation*}
                \dim(C_m \square C_n) = 
                \begin{cases}
                    3 & \text{if $m$ or $n$ is odd} \\
                    4 & \text{otherwise.}
                \end{cases}
            \end{equation*}
        \end{lemma}
        
        \begin{proposition} \label{prop: C2mxC4 >2}
            Let $C_{2p} \square C_4$ be a product of cycles with $p \geq 2$.
            Then $\zeta(C_{2p} \square C_4) > 2$.
        \end{proposition}
        \begin{proof}
            Assume that the Cop probes $B = \{b_1,b_2 \}$.
            Then there are only three types of probes:
            \begin{enumerate}
                \item[Type 1:] The projection of $B$ onto $C_4$ is a single vertex. 
                \item[Type 2:] In the projection of $B$ onto $C_4$, the vertices of the projection are adjacent. 
                \item[Type 3:] In the projection of $B$ onto $C_4$, the vertices of the projection are distance two apart. 
            \end{enumerate}
            It follows from \cref{lem: subset prod} that for probes of Type 1 and 3 that every column will contain a safe pair and from the structure of $C_4$ it follows that this will be a vertical safe pair $S_2^v$.
            Also from \cref{lem: subset prod} for a probe of Type 2 every column is resolved by the probe.
            Therefore, since $\dim(C_{2p} \square C_4)>2$ by \cref{lem: dim(CxC)} there must exist a safe pair.
            We will show that for a probe of Type 2, every two adjacent columns contain two diagonal safe pairs.
            
            Let $B$ be of Type 2. 
            Without loss of generality, assume that $b_1 = v_{i,3}$ and $b_2 = v_{i,2}$ for some column $i$.
            Now consider a column $k$ and say $\vec{D}(B,v_{k,3}) = [d_1,d_2]$.
            Then by \cref{lem: Bipart dist} and the structure of $C_4$, the distances from $B$ to the vertices in column $k$ are given in \cref{tab: B to column k}.
            \begin{table}[h]
                \centering
                \caption{The distances from $B=\{v_{i,3}, v_{i,2} \}$ to the vertices in column $k$ for $C_{2p} \square C_4$.}
                \begin{tabular}{cc}
                    $v$ & $\vec{D}(B,v)$ \\ \hline
                    $v_{k,3}$ & $[d_1,d_2]$ \\
                    $v_{k,2}$ & $[d_1+1,d_2-1]$ \\
                    $v_{k,1}$ & $[d_1+2,d_2]$ \\
                    $v_{k,0}$ & $[d_1+1,d_2+1]$
                \end{tabular}
                \label{tab: B to column k}
            \end{table}
            
            Now compare this to the distances from $B$ to the vertices in column $k+1$ as given in \cref{tab: B to column k+1}.
            The table gives all the possible distances to the vertices in column $k+1$, as it follows from \cref{lem: Bipart dist}.
            \begin{table}[h]
                \centering
                \caption{All possible distances from $B=\{v_{i,3}, v_{i,2} \}$ to the vertices in column $k+1$ for $C_{2p} \square C_4$ as by \cref{lem: Bipart dist}.}
                \begin{tabular}{ccccc}
                    $v$ & $\vec{D}(B,v)$ & $\vec{D}(B,v)$ & $\vec{D}(B,v)$ & $\vec{D}(B,v)$ \\ \hline
                    $v_{k+1,3}$ & $[d_1+1,d_2+1]$ & $[d_1+1,d_2-1]$ & $[d_1-1,d_2+1]$ & $[d_1-1,d_2-1]$ \\
                    $v_{k+1,2}$ & $[d_1+2,d_2]$ & $[d_1+2,d_2-2]$ & $[d_1,d_2]$ & $[d_1,d_2-2]$ \\
                    $v_{k+1,1}$ & $[d_1+3,d_2+1]$ & $[d_1+3,d_2-1]$ & $[d_1+1,d_2+1]$ & $[d_1+1,d_2-1]$ \\
                    $v_{k+1,0}$ & $[d_1+2,d_2+2]$ & $[d_1+2,d_2]$ & $[d_1,d_2+2]$ & $[d_1,d_2]$
                \end{tabular}
                \label{tab: B to column k+1}
            \end{table}
            
            It is clear from the tables that every two adjacent columns contain two diagonal safe pairs.
            We now consider two possibilities of a Robber set for the second turn:
            \begin{strat}[Diagonal safe pair] \label{strat: diag pair}
                If probe $B_2$ is of Type 1 or 3, a vertical safe pair will exist.
                This safe pair will either contain a vertex in row one and row three, or contain a vertex in row two and row four.
                It follows that the Robber can move to a safe pair in the next round.
                If probe $B_2$ is of Type 2, there is at least one other diagonal safe pair to move to. 
            \end{strat}
            \begin{strat}[Vertical safe pair] \label{strat: vert pair}
                As in the previous strategy, if $B_2$ is of Type 1 or 3 the Robber will either be safe or be able to move to a vertical safe pair.
                Hence assume that $B_2 = \{a_1, a_2 \}$ is of Type 2.
                If $a_1$ is in the same row as the Robber, then a diagonal safe pair will exist in columns $a_1$ and $a_1+1$ by \cref{tab: B to column k,tab: B to column k+1}.
                Otherwise if $a_1$ is not in the same row as the Robber, it again follows from \cref{tab: B to column k,tab: B to column k+1} that the Robber will be able to move to a diagonal safe pair. 
            \end{strat}
            
            The Robber can therefore perpetually avoid capture by using \cref{strat: diag pair} and \cref{strat: vert pair}.
        \end{proof}
        \begin{proposition} \label{prop: C2mxC4 <4}
            Let $C_{2p} \square C_4$ be the product of cycles with $p \geq 2$.
            Then $\zeta(C_{2p} \square C_4) \leq 3$.
        \end{proposition}
        \begin{proof}
            This already holds for $p \geq 4$ by Equation \cref{eq: cycle upbound}.
            Let $p=3$ and say the Cop probes $B_1 = \{v_{0,3}, v_{0,1}, v_{1,3} \}$ such that the distances to the vertices are given in \cref{tab: C6xC4 3 cops B1}.
            \begin{table}[h]
                \centering
                \caption{The distances $\vec{D}(B_1,v_{i,j})$ from probe $B_1$ to vertices $v_{i,j}$ in $C_6 \square C_4$.}
                \begin{tabular}{c|ccccccc}
                    $j=3$ & $[0,2,1]$ & $[1,3,0]$ & $[2,4,1]$ & $[3,5,2]$ & $[2,4,3]$ & $[1,3,2]$ \\
                    $j=2$ & $[1,1,2]$ & $[2,2,1]$ & $[3,3,2]$ & $[4,4,3]$ & $[3,3,4]$ & $[2,2,3]$ \\
                    $j=1$ & $[2,0,3]$ & $[3,1,2]$ & $[4,2,3]$ & $[5,3,4]$ & $[4,2,5]$ & $[3,1,4]$ \\
                    $j=0$ & $[1,1,2]$ & $[2,2,1]$ & $[3,3,2]$ & $[4,4,3]$ & $[3,3,4]$ & $[2,2,3]$ \\ \hline
                    & $i=0$ & $i=1$ & $i=2$ & $i=3$ & $i=4$ & $i=5$
                \end{tabular}
                \label{tab: C6xC4 3 cops B1}
            \end{table}
            
            From the table it can be seen that all safe sets have the form $\{v_{i,0}, v_{i,2} \}$ for $i \in \{0,1,\ldots,5 \}$.
            In the second turn, the Cop probes $B_2 = \{v_{i,0}, v_{i-1,1}, v_{i+1,1} \}$ such that $N[O_1]$ is resolved as shown in \cref{tab: C6xC4 3 cops B2}.
            \begin{table}[h]
                \centering
                \footnotesize
                \caption{The distances from the vertices in $N[O_1]$ to probe $B_2$ for $C_6 \square C_4$.}
                \begin{tabular}{c|cccccccc}
                    $v_{i,j} \in N[O_1]$ & $v_{i,0}$ & $v_{i-1,0}$ & $v_{i+1,0}$ & $v_{i,1}$ & $v_{i,2}$ & $v_{i-1,2}$ & $v_{i+1,2}$ & $v_{i,3}$ \\
                    $\vec{D}(B_2,v_{i,j})$ & $[0,2,2]$ & $[1,1,2]$ & $[1,3,1]$ & $[1,1,1]$ & $[2,2,2]$ & $[3,1,3]$ & $[3,3,1]$ & $[1,3,3]$
                \end{tabular}
                \label{tab: C6xC4 3 cops B2}
            \end{table}
            
            The case when $p=2$ follows in a similar fashion such that three cops are enough for $p \geq 2$.
        \end{proof}
        \Cref{prop: C2mxC4 >2,prop: C2mxC4 <4} together prove that $\zeta(C_{2p} \square C_4) = 3$.
        This completes all cases for $m$ and $n$ such that \cref{thm: cyclic grid loc} has been proved.

\section{Conclusions}
    In this paper, we showed that $\zeta( G \square H) \geq \max\{\zeta(G), \zeta(H)\}$ and that $\zeta(G \square H) \leq \zeta(G) + \psi(H) - 1$.
    We also showed that if $m=n=3$ or if $m$ is even while $n=4$, then $\zeta(C_m \square C_n)=3$ and that otherwise, $\zeta(C_m \square C_n)=2$.
    
    Note that cycles obtain both these lower and upper bounds.
    It would be worthwhile to investigate which other classes of graphs attain either the lower or the upper bound.

\bibliography{Article}
\end{document}